\documentclass[11pt]{article}  
\usepackage[a4paper]{geometry}
\usepackage{fullpage}
\usepackage[T1]{fontenc} % ou \usepackage[OT1]{fontenc}
\usepackage{hyperref}

\usepackage{amssymb,amsmath,amsthm}
\usepackage[english]{babel}
\usepackage{indentfirst}
\usepackage{graphicx}
\usepackage{enumerate}
\usepackage{appendix}
\usepackage{color}
\usepackage{racourci}

\newtheorem{theorem}{Theorem}[section]
\newtheorem{lemma}[theorem]{Lemma}
\newtheorem{definition}[theorem]{Definition}
\newtheorem{corollary}[theorem]{Corollary}
\newtheorem{proposition}[theorem]{Proposition}
\newtheorem{remark}[theorem]{Remark}

\title{Analysis of the factorization method for a general class of boundary conditions}
\author{Mathieu Chamaillard\footnote{Laboratoire POEMS, ENSTA ParisTech, 828, Boulevard des Maréchaux, 91762 Palaiseau Cedex, France}
 \and Nicolas Chaulet\footnote{Department of Mathematics, University College London, Gower Street, London, WC1E 6BT, United Kingdom}
  \and Houssem Haddar\footnote{INRIA Saclay Ile de France / CMAP Ecole Polytechnique, Route de Saclay, 91128 Palaiseau Cedex, France}}

\begin{document}

\maketitle
\date{}

\begin{abstract}
We analyze the factorization method (introduced by Kirsch in 1998 to solve inverse
scattering problems at fixed frequency from the far field operator) for a general class of boundary conditions
that generalizes impedance boundary conditions. For instance, when the surface impedance
operator is of pseudo-differential type, our main result stipulates that the
factorization method works if the order of this operator is different from
one and the operator is Fredholm of index zero with non negative imaginary
part.   We
also provide some validating numerical examples for boundary operators of
second order with discussion on the choice of the test function.  
\end{abstract}

\section{Introduction}
The factorization method is one of the most established inversion methods for
inverse scattering problems in which one is interested in reconstructing the shape
of inclusions from knowledge of the far field operator at a fixed frequency
in the resonant regime
\cite{Kir98,kirsch2008factorization}. The present work is a contribution to the analysis of the method for a large
class of boundary conditions verified on the boundary of the scattering
object. We consider here only the scalar case, the analysis of the method for electromagnetic vectorial
problems is still an open question.

Let us recall that this method provides a characterization of the obstacle shape
using the range of an operator explicitly constructed from the far field
data available at a fixed frequency, for all observation directions and all
plane wave incident directions.  Since this characterization (and subsequent
algorithm) is independent
of the boundary conditions, a natural question would be to specify the class
of boundary conditions for which the method works. We remark that as a corollary, one also 
obtains a uniqueness result for the shape reconstruction without knowledge of
the boundary conditions (in the considered class). 

 Motivated by so-called
generalized impedance boundary conditions that model thin layer effects (due to
coatings, small roughness, high conductivity, etc...) \cite{bookSV95,BenLem96,DHJ06-S,HJ02-201,HJN05-1273}, we consider here
the cases of impedance boundary conditions where the impedance corresponds to a boundary operator acting on a
Hilbert space with values in its dual. Pseudo-differential boundary operators
can be seen as a particular case of this setting. The simplest form (classical
impedance boundary conditions) corresponds to
multiplication by a function  and has been analyzed in \cite{kirsch2008factorization}, whereas for non linear
inversion methods associated with second order operators we refer to \cite{BoChHa12,krcaip12}.

We first analyze the forward problem using a surface formulation of the
scattering problem. This method allows us to show well posedness of the
forward scattering problem under weak assumptions on the boundary operator. We then analyze the factorization method for this class
of operators. We demonstrate in particular that the method is valid for  boundary operators
``of order'' (if we see them as surface pseudo-differential operators) strictly
less than one or strictly greater than one, which are of Fredholm type with
index zero and have non negative imaginary parts. We observe that
in the first case (order < 1), the analysis of the method follows the lines of the case of
classical impedance boundary conditions as in \cite{kirsch2008factorization} while the second case is
rather similar to the case of Dirichlet boundary conditions. Our
proof however fails  in considering boundary operators of order 1 or mixtures
of operators of order less than one and operators of order greater than one. We face here the same
difficulty as one encounters in trying to answer the open question related to
the factorization method for obstacles
with mixture of Neumann and Dirichlet boundary conditions. 

For the numerical two dimensional tests, we rely on the use of classical Tikhonov-Morozov
regularization strategy to solve the factorization equation.  Our numerical
experience suggested that this method handles the case of noisy data more
easily than the
classically used spectral truncation method. We refer however to \cite{Lec06} for tricky
numerical investigations on this type of regularization in the case of noisy data. 
We also use combinations of monopoles and dipoles as test functions for implementing the
inversion algorithm. Numerical examples show how this provides better
reconstructions as compared with the use of 
monopole test functions only.

The outline of the paper is as follows. The second section is dedicated to the
introduction of the forward problem for a general class of impedance boundary
conditions and we explain the principle of the factorization method. Section 3 is devoted
to the analysis of the forward scattering problem under general assumptions on
the boundary operator, with some applicative examples motivated by thin coating
models. The analysis of the factorization method is done in section 4. The last
section is devoted to numerical implementation of the inversion algorithm and
some validating results.  An appendix is added for the proof of two technical
results.

% section 2
\section{Quick presentation of the forward problem and the factorization method}
Let $D$ be an open bounded domain  of $\bR^{d}$ for $d=2,3$ with Lipschitz boundary $\Gamma$. Assume in addition that its complement $\Om:=\bR^d\setminus\overline{D}$ is connected. The scattering of an incident wave $u_i$ that is solution to the homogeneous Helmholtz equation in $\bR^{d}$ by an obstacle $D$ gives rise to a field $u_{s}$ that is solution to
\begin{equation}
	\begin{cases}\Delta u_s+k^2 u_s=0 \text{ in } \Om,\\
	\pn u_s+\Zimp u_s=-\left( \pn u_i + \Zimp u_i\right) \sg,\\
	\end{cases}
	\label{pbdiffdeb}
\end{equation}
and $u_s$ satisfies the radiation condition
\begin{equation}
\label{eq:radiation}
\displaystyle\lim_{R \to \infty} \int_{|x|=R} | \partial_r u_s -i k u_s |^2 ds=0.
\end{equation}
Here $\nu$ is the unit outward normal to $\Gamma$ and $\Zimp$ is an impedance surface operator that can possibly depend on the wavenumber $k$. In the next section we give a precise definition of $\Zimp$ (see Definition \ref{DefZ}) and we prove that under appropriate assumptions on $\Zimp$ the forward problem is well-posed, i.e. it has a unique solution in a reasonable energy space and this solution depends continuously on the incident wave. In the following, we say that $u_s$ is an outgoing solution to the Helmholtz equation if it satisfies the radiation condition \eqref{eq:radiation}.
In this case, one can uniquely define the far field pattern  $u^\infty:S^{d-1} \mapsto \mathbb{C}$ associated with the scattered field $u_{s}$  (see \cite{colton1998inverse} for further details on scattering theory) by
\begin{equation*}
 u_s(x)=u^{\infty}(\hat{x}) \gamma(d)\frac{e^{ik|x|}}{|x|^{\frac{d-1}{2}}}\left(1 +  O\left(\frac{1}{|x|}\right)\right) ,\ |x| \to \infty,
\end{equation*}
uniformly for $\hat{x}:=x/|x| \in S^{d-1}$ where $S^{d-1}$ is the unit sphere of $\bR^{d}$ and $\gamma(2):=\frac{e^{\frac{i \pi}{4}}}{\sqrt{8 \pi k}}$ and $\gamma(3):=\frac{1}{4 \pi}$.
Let us recall that for  any bounded set $\mathcal{O}$ of Lipschitz boundary $\partial \mathcal{O}$ such that $D \subset \mathcal{O}$, the following representation formula holds:
\begin{equation*}
%\label{representation_far field}
 u^{\infty}(\hat{x})= \int_{\partial \mathcal{O}} \left(u_s(y) \frac{\partial e^{-ik\hat{x} \cdot y}}{\partial \nu(y)}
-e^{-ik\hat{x} \cdot y} \frac{\partial u_s(y)}{\partial \nu(y)} \right)ds(y),\ \hat{x} \in S^{d-1}.
\end{equation*}
We define $u^{\infty}(\cdot,\hat{\theta})$ the far field pattern associated with the  scattering of an incident plane  wave of incident direction $\hat{\theta}$, which is given by $u_i(x)=e^{ik \hat{\theta} \cdot x}$. The inverse problem we consider is then the following:  from  knowledge of $\{u^\infty(\hat{x},\hat{\theta}) \text{  for all } (\hat{x},\hat{\theta}) \in S^{d-1} \times S^{d-1}\}$ find the obstacle $D$ with minimal \textit{a priori} assumptions on $\Zimp$. 

Let us introduce $F:L^2(S^{d-1})\mapsto L^2(S^{d-1})$, the so-called far field operator defined for $g\in L^{2}(S^{d-1})$ by
\begin{equation}
\label{far fieldOp}
  (Fg)(\hat{x})=\int_{S^{d-1}} u^\infty(\hat{x},\hat{\theta}) g(\hat{\theta})ds(\hat{\theta}),\qquad \text{for } \hat x \in S^{d-1}.
\end{equation}
This operator factorizes in a particular way, namely there exists two bounded operators $G$ and $T$, defined in section \ref{sec:formal}, such that
\[
F = -GT^*G^*.
\] 
As we show in Lemma \ref{LeImG} the operator $G$ characterizes the obstacle $D$ and we have the equivalence
\[
z \in D \Longleftrightarrow \phi_z\in \mc{R}(G) 
\]
for $\phi_z(\hat{x}):= e^{-ik\hat{x}\cdot z}$ and where for any bounded linear operator $L$ defined from an Hilbert space $E_1$ into another Hilbert space $E_2$, the space $\mc{R}(L) := \{ y \in E_2\;|\; \exists\, x \in E_1  \text{ s.t. } y=Lx \}$ is the range of $L$. As shown in section \ref{sec:equi}, the Factorization Theorem  (see \cite[Theorem 2.15]{kirsch2008factorization})  applies and gives the equality between the ranges of $F_\#^{\frac{1}{2}}$ and $G$ where $F_\#$ is the self-adjoint and positive operator given by
\[
F_{\#}= |Re(F )| + Im(F ).
\]
See (\ref{defreim}) for a definition of the real and imaginary parts. As a consequence, the far field patterns for all incident plane waves are sufficient to determine the indicator function of the scatterer $D$ thanks to the equivalence
\[
z \in D \Longleftrightarrow \phi_z\in \mc{R}(F_\#^{\frac{1}{2}}) .
\]
Remark that this equivalence does not depend on the impedance operator
$\Zimp$. Hence, if justified, the factorization method gives a simple way to
compute the indicator function independently from the impedance operator.

% section 3
\section{Analysis of the forward problem for a general class of impedance operators} 
\label{sec:direct}
\subsection{Mathematical framework}
\label{sec:notations}
Let $\vgamma \subset \ldeux$ equipped with the scalar product $(\cdot,\cdot)_{\vgamma}$ be a Hilbert space and assume that the embedding $  \cinfty \subset \vgamma$ is continuous and dense.  For any Hilbert space $X$ we denote by $X^*$ the space of continuous anti-linear forms on $X$ and we define the duality pairings for all $(u,v) \in X \times X^*$ by
\begin{equation*}
    %\label{duality}
    \langle v,u \rangle_{X^*,X}:=v(u),\quad
    \langle u,v \rangle_{X,X^{*}}:=\overline{v(u)}.
\end{equation*}
From now on, if not needed, we  do not specify the spaces for the duality
products. We then define an impedance operator on $\vgamma$ as:
\begin{definition}
\label{DefZ}
For any frequency $k \in \bC$ an impedance operator $\Zimp$ is a linear bounded  operator from $\vgamma$ into $\vgamma^*$.
\end{definition}
\noindent In this case, equations (\ref{pbdiffdeb})-\eqref{eq:radiation} make  sense for any $u_{s}$ in 
\begin{equation*}\hlocv:=\soboloc \end{equation*}
where $\gamma_0$ is the trace operator on $\Gamma$. Problem
(\ref{pbdiffdeb})-\eqref{eq:radiation} can be seen as a particular case of: find  $u_s \in \hlocv$ such that
\begin{equation}
	\label{pbdiff}
	\begin{cases}\Delta u_s+k^2 u_s=0 \rnd,\\
	\pn u_s+\Zimp u_s=f \sg,\\	
	\displaystyle\lim_{R \to \infty} \int_{|x|=R} | \partial_r u_s -i k u_s |^2 ds=0
	\end{cases}
\end{equation}
with $f:=-(\pn u_i+\Zimp u_i)$. In fact, this problem makes sense whenever the right-hand side $f$ is in the dual space $\espacedual$ of the Hilbert space $  \espace:=\vgamma \cap \sobolev{\frac{1}{2}}$ equipped with the scalar product $  (\cdot, \cdot)_{\espace}:=(\cdot, \cdot)_{\vgamma}+(\cdot, \cdot )_{\trace} $. The aim of this section is to determine conditions on $\Zimp$ under which \eqref{pbdiff} is well posed in the sense that for all $f\in \espacedual$, it has a unique solution  in $\hlocv$ and  for any  bounded open subset $K$ of $\Om$, there exists a constant $C_K>0$ such that
\begin{equation*}
 % \label{contipbdiff}
  \|u_s\|_{V(\Gamma)}+\|u_s\|_{H^1(K)} \leq C_K \|f\|_{\espacedual}.
\end{equation*}
To do so we first assume that $\Zimp$ satisfies
\begin{equation} \label{H1}
\Im(\Zimp) \geq 0,
\end{equation}
which corresponds to the physical assumption that the obstacle $D$ does not produce
energy. 
% \begin{itemize}
% \label{HypothesisZ}
%  \item[(H1)] $ \Im(\Zimp) \geq 0 \label{posim}$,
% \item[(H2)] the operator $\Zimp$ is symmetric : $ \forall d \in V ,  \Zimp^*d=\overline{\Zimp\overline{d}}$.
%  \end{itemize}
For all Hilbert space $E$ and all linear and continuous operators $T:E \mapsto E^*$ its adjoint $T^*:E \mapsto E^*$ is defined  by $ \langle T^*u,v\rangle_{E^*,E}=\langle u,Tv\rangle_{E,E^*} $  for $(u,v) \in E \times E$. The real and imaginary parts  of $T$ are then given by
\begin{equation}
\label{defreim}
    \Re(T):=\cfrac{T+T^*}{2} \quad \text{and} \quad    \Im(T):=\cfrac{T-T^*}{2i}.
\end{equation}
Finally, since $\cinfty \subset \vgamma$ then the embedding of $\espace$   into $\trace$ is dense, thus we have the following dense and continuous embeddings:
\begin{equation*}
  \vgamma \subset \ldeux \subset \vgamma^*
\end{equation*}
and
\begin{equation*}
  \espace \subset \trace \subset \ldeux \subset \traced \subset \espacedual.
\end{equation*}

\subsection{An equivalent surface formulation of the problem}
In order to optimize the hypothesis on $\Zimp$, we propose to rewrite the problem (\ref{pbdiff}) as a surface problem on the boundary $\Gamma$. To do so we use the so-called exterior Dirichlet-to-Neumann map $n_e:\trace \mapsto \traced$ defined for $h \in \trace$ by $n_e h=\pn u_{h}$
where $u_{h} \in \hloc:=\sobolo$ is the unique solution to
\begin{equation}
	\begin{cases}\Delta u_{h}+k^2 u_h=0 \rnd,\\
	 u_h=h \sg,\\	
	\displaystyle\lim_{R \to \infty} \int_{|x|=R} | \partial_r u_h -i k u_h |^2 ds=0.
	\end{cases}
	\label{dirclassique}
\end{equation}
 For $f$ in $\espacedual$, problem \eqref{pbdiff} then writes: find $u_{\Gamma} \in \espace$ such that
 \begin{equation}
\label{problemsurf}
 \Zimp u_{\Gamma}+n_e u_{\Gamma}=f
\end{equation}
and this formulation is equivalent to \eqref{pbdiff}. The surface formulation
that we adopt here is similar to the one used in \cite{vernhet} for the study of
second order surface operators.
\begin{lemma}
  \label{lem:reformulation}
 For $f$ in $\espacedual$, if $u_s \in \hlocv $ is a solution of \eqref{pbdiff} then $\gamma_0 u_s \in \espace$ is a solution of \eqref{problemsurf}. Conversely  if $u_\Gamma \in \espace$ is a solution of \eqref{problemsurf} then the unique  solution $u_s \in \hloc$ of \eqref{dirclassique} with $h=u_\Gamma$ is a solution of \eqref{pbdiff}.
\end{lemma}
\begin{proof}
Take $f \in \espacedual$, if $u_s \in \hlocv$ is a solution of \eqref{pbdiff} then $\gamma_0 u_s$ belongs to $\espace$. Moreover
  \begin{equation*}
    \pn u_s=n_e(\gamma_0 u_s)
  \end{equation*}
  since  $u_s$ is an outgoing solution to the Helmholtz equation outside $D$  and therefore $\gamma_0 u_s $ is a solution of \eqref{pbdiff}.

Conversely assume that $u_\Gamma \in \espace$ is a  solution of \eqref{problemsurf} and define $u_s \in \hloc$  the unique solution of \eqref{dirclassique} with $h=u_\Gamma$. Thus $u_s \in \hlocv$ is a solution of \eqref{pbdiff} because we have:
  \begin{equation*}
    \pn u_s=n_e u_\Gamma = f - \Zimp u_s.
  \end{equation*}
\end{proof}

\subsection{Existence and uniqueness  of solutions}
To prove existence and uniqueness of the solution to problem \eqref{pbdiff} we use the surface formulation \eqref{problemsurf} and prove that it is of Fredholm type. Then the following Lemma on uniqueness is sufficient to also have existence and continuous dependance of the solution with respect to the right-hand side.
\begin{lemma}
\label{uniqueness_lem}
The operator $\Zimp+n_e:\espace \mapsto \espacedual$ is injective.
\end{lemma}
\begin{proof}
Thanks to Lemma \ref {lem:reformulation} it is sufficient to prove uniqueness for the volume  equation \eqref{pbdiff}.
Let $u_s$ be a solution of \eqref{pbdiff} with $f=0$ then we have
\begin{equation*}
  \langle \pn u_s,u_s \rangle+\langle\Zimp u_s,u_s\rangle=0.
\end{equation*}
Hypothesis \eqref{H1} implies:
\begin{equation}
  \label{uniqneg}
  \Im \langle \pn u_s, u_s \rangle \leq 0.
\end{equation}
Let us denote by $B_r$ an open ball of radius $r$ which contains $\overline{D}$.  A Green's formula in $\Omega_r:=B_r\setminus \overline{D}$ yields:
\begin{equation*}
  \Im \left( \int_{\partial B_r} \partial_r u_s \overline{ u_s } ds \right)=\Im \langle \pn u_s,u_s\rangle \leq 0,
\end{equation*}
 by \eqref{uniqneg}. Thanks to the Rellich lemma \cite{colton1998inverse},  $u_s=0$ outside $B_r$ and by unique continuation we have  $u_s=0$ in $\Om$. 
\end{proof}
It still remains to give some hypothesis under which $\Zimp+n_{e}$ is of
Fredholm type. To do so, we need the following fundamental properties of the Dirichlet-to-Neumann map.
\begin{proposition}
\label{fredne}
There exists a bounded operator $C_e$ from $ \trace$ into $\traced$ 
which satisfies the following coercivity and sign conditions: $ \exists c_e>0$ such that for all $x \in \trace$:
\begin{equation*}
 \begin{cases}
\Re\langle C_e x,x \rangle \geq c_e \| x\|^2_{ \sobolev{\frac{1}{2}}} ,\\
\Im\langle C_e x,x \rangle \leq 0,
\end{cases}
\end{equation*}
and a compact operator $K_e$~ from $\trace $ into $ \traced$ such that:
\begin{equation*}
 n_e=-C_e+K_e.
\end{equation*}
\end{proposition}
\noindent The proof of this Proposition is classical and is given in the appendix for the
reader convenience. 
\begin{corollary}
\label{existence}
If one of the following assumptions is fulfilled
\begin{enumerate}
\item  the operator $\Zimp$ writes as $-C_Z+K_Z$ where $C_Z:\espace \mapsto \espacedual$ satisfies:
\begin{equation}
\label{eq:poscz}
 \Re\langle C_Z x,x \rangle - \Im \langle C_Z x,x \rangle  \geq c_z \| x\|^2_\vgamma \ \forall x \in \vgamma,
\end{equation}
for $c_z>0$ independent of $x$ and  $K_Z:\espace \mapsto \espacedual$ is a compact operator,
\item the space $\vgamma$ is compactly embedded into $\trace$ and $\Zimp:\vgamma\mapsto\vgamma^*$ writes as $T_Z+K_Z$ for some isomorphism $T_Z:\vgamma \mapsto \vgamma^*$ and compact operator
$K_Z:\vgamma \mapsto \vgamma^*$, 
\item the space $\trace$ is compactly embedded into $\vgamma$,
\end{enumerate}
    then  $\Zimp+n_e$ is an isomorphism from $\espace$ into $\espacedual$.
\end{corollary}
\begin{proof}
Assume that the first assumption is fulfilled. Thanks to  Proposition \ref{fredne}, there exists a compact operator $K_e:\trace \mapsto \traced$ and  an operator $C_e:\trace \mapsto \traced$  with coercive real  part such that $n_e = -C_e+K_e$. For all $z$ in $\mathbb{C}$ we have
\begin{equation*}
  | z | \geq \frac{|\Re(z)|+|\Im(z)|}{\sqrt{2}}, 
\end{equation*}
and then for all $x \in \espace$
\begin{align*}
 \left|\langle (C_e+C_Z)x,x \rangle\right|& \geq \cfrac{|\Re\langle (C_e +C_Z) x,x \rangle| + |\Im\langle (C_e+C_Z) x,x\rangle|}{\sqrt{2}} \\
 &\geq \cfrac{
\Re\langle C_e x,x \rangle+ \Re\langle C_Z x, x \rangle- \Im\langle C_Z x,x\rangle-\Im\langle C_e x,x\rangle}{\sqrt{2}},
\end{align*}
and from  \eqref{eq:poscz}
\begin{equation*}
 \left|\langle (C_e+C_Z)x,x \rangle\right|\geq
\cfrac{\Re\langle C_e x,x \rangle+ c_z\|x\|^2_\vgamma-\Im\langle C_e x,x\rangle}{\sqrt{2}}.
\end{equation*}
Proposition \ref{fredne} yields:
\begin{equation*}
\left|\langle (C_e+C_Z)x,x \rangle\right|\geq
\cfrac{c_e \|x\|^2_\trace+ c_z \|x\|^2_\vgamma}{\sqrt{2}},
\end{equation*}
and  then $C_e+C_Z:\espace \mapsto \espacedual$ is coercive on $\espace$. Furthermore $K_e+K_Z$  is compact from  $\espace$ into $\espacedual$ and
\begin{equation*}\Zimp+n_e=-(C_e+C_Z)+K_e+K_Z \end{equation*}
which implies that $\Zimp+n_e$ is a Fredholm type operator of index zero. The conclusion is a direct consequence of the uniqueness Lemma \ref{uniqueness_lem}.

If we suppose that the second assumption is fulfilled,  then $V(\Gamma) = \espace$ and they have equivalent norms. In addition, $n_e:\espace \mapsto \espacedual$ is a compact operator and then  $\Zimp+n_e:\espace \mapsto \espacedual$ is clearly a Fredholm type operator of index zero. 

Finally if we assume that the last assumption is fulfilled then $\trace = \espace$ and they have equivalent norms. Thanks to Proposition \ref{fredne} there exist an isomorphism $C_e:\trace \mapsto \traced$ and a compact operator $K_e:\trace \mapsto \traced$ such that  $ n_e=-C_e+K_e$. Since $\trace$ is compactly embedded into $\vgamma$, the operator $\Zimp:\trace \mapsto \traced$ is compact and so is  $K_e+\Zimp:\trace \mapsto \traced$. 
Thus   $\Zimp+n_e:\trace \mapsto \traced$ is a Fredholm type operator of index zero and by Lemma \ref{uniqueness_lem} it is an isomorphism.
\end{proof}

\begin{remark}[Example of applications]
\label{sec:exemple_application}
We denote by $\gradt$ the surface gradient operator on $\Gamma$, $\divt$  the surface divergence operator on $\Gamma$ which is the $L^2$ adjoint of $-\gradt$ and 
$\Delta_\Gamma:=\divt\gradt$ the Laplace Beltrami operator on $\Gamma$ (see \cite{Ned01} for more details on these surface operators). 
We consider
\[\Zimp=\Delta_\Gamma \nudelta \Delta_\Gamma - \divt \mudelta \gradt
+\lambdadelta\] with $(\nudelta,\mudelta,\lambdadelta) \in
(L^{\infty}(\Gamma))^3$ such that the imaginary parts of the bounded functions
$\nu_k$, $\mu_k$ and $\lambda_k$ are non negative and such that the real part
of $\nu_k$ is positive definite.
This example is motivated by high order asymptotic models  associated with thin
coatings or imperfectly conducting obstacles (see \cite{BenLem96,HJN05-1273}). One can easily check that classical variational
techniques cannot let us conclude on the analysis of the scattering
problem associated with $\Zimp$ while the method developed above, namely point
3 of Corollary \ref{existence}, indicates that
this problem is well posed with
$\vgamma:=\sobolev{2}$.

\end{remark}

\section{Factorization of the far field operator}
We shall assume here  that one of  the hypothesis of Corollary \ref{existence}
is fulfilled so that  $\Zimp + n_e : \espace \mapsto \espacedual $ is an
isomorphism and problem \eqref{pbdiff} is well posed.
\subsection{Formal factorization}
\label{sec:formal}
In this part we give a formal factorization of the far field operator  defined by \eqref{far fieldOp} under the form $F = -GT^{*}G^{*}$. To achieve this objective, we proceed as for classical impedance boundary conditions. Let $G,\traceH$ and $\partialH$ be defined by $Gf:=u^\infty$ where $u^\infty$ is the far field associated with the solution $u_s$ of \eqref{pbdiff}, $\traceH g:=\gamma_0 v_g$ and  $\partialH g= \gamma_1 v_g$ where for $g\in L^2(S^{d-1})$,
\[
v_{g}(x):= \int_{S^{d-1}} e^{ik\hat{\theta}\cdot x} g(\hat{\theta}) \,ds(\hat{\theta}).
\] 
The operator $\gamma_1$ denotes the normal derivative trace operator defined for $v \in \{ w\in \hloc \text{ s.t. } \Delta w \in L^2(\Om)\}$ by
\[
\gamma_1 v := \partial_\nu v|_\Gamma.
\]
We recall that $F : L^2(S^{d-1}) \rightarrow L^2(S^{d-1})$ is defined for $g \in L^2(S^{d-1})$ by
\[
(Fg)(\hat{x}) = \int_{S^{d-1}} u^\infty(\hat{x},\hat{\theta}) g(\hat{\theta}) ds(\hat{\theta})
\]
where $u^\infty(.,\hat{\theta})$ is the far field pattern associated with the incident wave $x \mapsto e^{ik \hat{\theta}\cdot x}$. For $(\hat \theta,\hat x) \in (S^{d-1})^2$, this far field writes as
$u^\infty(\hat x,\hat{\theta}) = -[G(\Zimp e^{ik \hat{\theta}\cdot x}+\gamma_1 e^{ik \hat{\theta}\cdot x})](\hat x)$. Hence, by linearity of the forward problem with respect to the incident waves, 
\begin{align*}
 ( F g)(\hat x)=\int_{S^{d-1}}-[G(\Zimp e^{ik \hat{\theta}\cdot x}+\gamma_1e^{ik \hat{\theta}\cdot x})](\hat x)g(\hat{\theta})ds(\hat{\theta})=
-G\left(\Zimp v_g+\gamma_1 v_g\right)(\hat x)
\end{align*}
and then we get the first step of the formal factorization:
\begin{equation}
  \label{eq:premier_decomposition}
  F=-G(\Zimp\traceH+\partialH).
\end{equation}
Formally, the adjoints of $\traceH$ and $\partialH$  are respectively given for all $q$ and $\hat{x} \in S^{d-1}$ by:
\begin{equation*}
 % \label{eq:adjHh}
    (\traceH^* q)(\hat{x})=\displaystyle \int_{\Gamma} e^{-ik \hat{x}y}q(y) ds(y),\quad
    (\partialH^* q)(\hat{x})=\displaystyle\int_{\Gamma} \frac{\partial e^{-ik \hat{x}y}}{\partial \nu(y)}q(y)ds(y).
\end{equation*}
Therefore,
\begin{equation}
 \label{eq:faradj}
  (\Zimp\traceH+\partialH)^*q(\hat{x})=\displaystyle \int_{\Gamma} e^{-ik \hat{x}y}\Zimp^*q(y)ds(y)+\int_{\Gamma} \frac{\partial e^{-ik \hat{x}y}}{\partial \nu(y)}q(y)ds(y).
\end{equation}
We recognize that the latter is the far field pattern associated with
\begin{equation}
 \label{eq:defv}
 v:= \SL(\Zimp^*q)+\DL(q),
\end{equation}
where $\SL$ et $\DL$ are the single and double layer potentials associated with the wave number $k$:
\begin{equation*}
 \begin{cases}
  \SL(q)(x)=\displaystyle\int_{\Gamma} G_{k}(x-y)q(y)ds(y) ,\ x \in \mathbb{R}^d\setminus{\Gamma}, \\
  \DL(q)(x)=\displaystyle\int_{\Gamma} \frac{\partial  G_{k}(x-y)}{\partial \nu(y)}q(y)ds(y),\ x \in \mathbb{R}^d\setminus{\Gamma},
 \end{cases}
\end{equation*}
and $G_{k}$ is the Green's outgoing  function of $\Delta+k^2$ given for $x \in \mathbb{R}^d$ by
\begin{equation*}
  G_{k}(x)=    \frac{i}{4}H_0^{1}(k|x|) \;  \text{ if } d=2; \qquad 
    G_{k}(x) = \frac{e^{ik|x|}}{4 \pi|x|}  \text{ if } d = 3, 
\end{equation*}
where $H_0^{1}(k|x|)$ is the Hankel function of the first kind and order zero.
The function $v$ defined by \eqref{eq:defv} is a solution of problem \eqref{pbdiff} with the right hand-side
\begin{align*}
  f&=\Zimp v +\pn v=\Zimp\traceps \Zimp^*q+\normalepd q + \Zimp \tracepd q +\normaleps \Zimp^* q, \end{align*}
where
\begin{equation*}
  \begin{matrix}
    \traceps&:=\gamma_0 \SL,&
    \tracepd&:=\gamma_0\DL,\\
    \normaleps&:=\gamma_1 \SL,&
    \normalepd&:=\gamma_1 \DL.    
  \end{matrix}
\end{equation*}
Thus we can deduce that
\begin{equation}
 \label{decomposition_hstart}  
  (\Zimp \traceH+\partialH)^*=GT,
\end{equation}
with
\begin{equation*}
  T:=\Zimp\traceps \Zimp^*+\normalepd + \Zimp \tracepd +\normaleps \Zimp^*
\end{equation*}
and from \eqref{eq:premier_decomposition} we finally get
\begin{equation}
  \label{eq:facto}
  F=-G T^*G^*.
\end{equation}

\subsection{Appropriate function setting}
Thanks to \eqref{eq:facto}, the operator $F$ has the form required by the
Factorization Theorem \cite[Theorem 2.15]{kirsch2008factorization} that allows us to prove that $D$ is characterized by the range of $F_{\#}^\frac{1}{2}$.
We just have to find two Hilbert spaces $X  $ and $Y$ such that the operators:
\begin{align*}
  G:X &\longmapsto Y,\\
  T^*:X^*& \longmapsto X,
\end{align*}
are bounded. The space $X$ has to be a reflexive Banach space such that $X^* \subset U \subset X$ with continuous and dense embedding for some Hilbert space $U$. The space $Y$ being the data space, the choice $Y=L^2(S^{d-1})$ seems appropriate. For the space $U$ we take $L^2(\Gamma)$, and given that $G$ maps a right-hand side $f$ to the far field pattern  of the solution to \eqref{pbdiff}, $\espacedual$ is a natural choice for  $X$. Thus $T$ has to be bounded from $\espace$ onto $\espacedual$ which is false in general since $\Zimp\traceps \Zimp^*:\espace \mapsto \espacedual$ is not necessarily bounded. This is in particular the case for the impedance operator $\Zimp = \Delta_\Gamma$ which is continuous from $\vgamma=H^1(\Gamma)$ into $H^{-1}(\Gamma)$. To overcome this difficulty, let us introduce the following Hilbert space:
\begin{equation*}
 \esplambda:=\left\{ u \in \espace ,\ \Zimp^*u \in \sobolev{-\frac{1}{2}} \right\}
\end{equation*}
associated with the scalar product
\begin{equation*}
(u,v)_\esplambda=(u,v)_{\espace}+(\Zimp^*u,\Zimp^*v)_{\traced}.
\end{equation*}
We denote by $\|\cdot\|_{\esplambda}$ the norm associated with its scalar product and this norm  is equivalent to the norm
\begin{equation*}
  ||| \cdot |||^2_{\esplambda} := \| \cdot \|^2_{\trace} + \|\Zimp^* \cdot\|^2_{\traced}.
\end{equation*}
Indeed, since $\Zimp+n_e:\espace \mapsto \espacedual$ is an isomorphism, $(\Zimp+n_e)^*:\espace \mapsto \espacedual$ is also an isomorphism and there exists $C>0$ such that for all $v$ in $\esplambda$
\begin{align*}
  \|v\|_{\trace} \leq C \| (\Zimp+n_e)^* v\|_{\espacedual} \leq C \| (\Zimp+n_e)^* v\|_{\traced}\\ \leq C \left( \|v\|_\trace+\|\Zimp^* v\|_{\traced} \right),
\end{align*}
and then there exists $C>0$ such that for all $v \in \esplambda$:
\begin{equation*}
  \frac{1}{C}||| v |||_{\esplambda} \leq \| v \|_{\esplambda} \leq C ||| v |||_{\esplambda}.
\end{equation*}
 \begin{remark}
  \label{rem:inclusion}
  If we have the inclusion $\trace \subset \vgamma$ then:
  \begin{equation*}
    \espace=\trace \quad \text{ and } \quad \esplambda=\trace.
  \end{equation*}
\end{remark}
With this space, we recover some symmetry for the spaces on which $T$ is defined and the factorization \eqref{eq:facto} makes sense as proven in the remainder of this section.
But first of all we indicate an important property of $\esplambda$. 
\begin{lemma}
\label{Lambdabien}
 The space $\esplambda$ is  dense in $\espace$ and
\begin{equation*}
 	\esplambda \subset \espace \subset \sobolev{\frac{1}{2}} \subset L^2(\Gamma) \subset  \sobolev{-\frac{1}{2}} \subset \espacedual \subset \esplambda^*.
\end{equation*}
Moreover, if $\vgamma$ is compactly embedded in $\trace$ then $\esplambda$  is compactly embedded in $\espace$.
\end{lemma} 
\noindent The proof of this lemma is straightforward and is given in the appendix
for readers' convenience.
\begin{proposition}
\label{extension}
 The operator $\Zimp+n_e$ can be extended to an isomorphism from $\trace$ into $\esplambda^*$.
\end{proposition}
\begin{proof}
First, let us prove that  $(\Zimp+n_e)^*$ is an  isomorphism from $\esplambda$ into $\traced$. By definition of $\esplambda$ we have:
\begin{align}
 \esplambda&=\left\{ u \in \espace ,\ \Zimp^*u \in \traced \right\}, \notag\\
	&=\left\{ u \in \espace,\ (\Zimp+n_e)^*u \in \traced \right\},\notag\\
\label{eq:ensemble}
&=\left((\Zimp+n_e)^*\right)^{-1}\left(\traced\right).
\end{align}
Since $(\Zimp+n_e)^*$ is assumed to be an isomorphism from $\espace$ into $\espacedual$ and $\traced \subset \espacedual$, \eqref{eq:ensemble} implies that $(\Zimp+n_e)^*:\esplambda \mapsto \traced$ is bijective.  By definition of the norm on  $\esplambda$ we get the continuity of $(\Zimp+n_e)^*$ defined from $\esplambda$ into $\traced$.  Then $(\Zimp+n_e)^*:\esplambda \mapsto \traced$ is an isomorphism and its adjoint $(\Zimp+n_e):\trace \mapsto \esplambda^*$ is also an isomorphism.
\end{proof}
Then the following Lemma is straightforward and completes the factorization of $F$.
\begin{lemma}
\label{LeT}
The operators $G$ and $T$ have the following properties,
\begin{itemize}
\item the operator $G$ is bounded from $\esplambda^*$ to $L^2(S^{d-1})$,
\item the operator $T$ is bounded from $\esplambda$ to $\esplambda^*$.
\end{itemize}
\end{lemma}
\begin{proof}
Let us introduce the bounded operator $A^{\infty}:\trace \mapsto L^2(S^{d-1})$ defined by $A^{\infty}h=u^{\infty}$ where $u^{\infty}$ is the far field pattern associated with the scattered field $u_h \in \hloc$ solution to the Dirichlet problem \eqref{dirclassique}.  We have: 
\begin{equation}
\label{factoG}
 G=A^{\infty}(\Zimp+n_e)^{-1}
\end{equation}
and Proposition \ref{extension} proves the first point. 

Let us recall the definition of $T$:
\begin{equation*}
  T=\Zimp\traceps \Zimp^*+\normalepd + \Zimp \tracepd +\normaleps \Zimp^*.
\end{equation*}
The definition of $\esplambda$ implies that $\Zimp:\trace \mapsto \esplambda^*$ and $\Zimp^*:\esplambda \mapsto \traced$ are two bounded operators. Thus using that $\traceps:\traced \mapsto \trace$, $\normalepd:\trace \mapsto \traced$, $\tracepd:\traced \mapsto \trace$ and $\normaleps:\traced \mapsto \traced$ are bounded (see \cite[chapter 7]{McL00}) we obtain the second point. 
\end{proof}
We show in the next section how to use the factorization \eqref{eq:facto} to determine the obstacle $D$.

\subsection{Characterization of the obstacle using $F_\sharp$}
\label{sec:equi}
First of all, $D$ is characterized by the range of $G$ as stated in the next Lemma. 
\begin{lemma}
\label{LeImG}
Let $\phi_z \in L^2(S^{d-1})$ be defined for $(z,\hat{x}) \in  \mathbb{R}^d \times S^{d-1}$ by $\phi_z(\hat{x})=e^{-ik\hat{x}z}$  
then we have
\begin{equation*}
	%\label{equirang}
  z \in D \Longleftrightarrow \phi_z \in \mathcal{R}(G).
\end{equation*}
\label{rangphiz}
\end{lemma}
\begin{proof}
 If $z\in D$, then for $f(x)=G_k(x-z)$ we have 
 \[
 \phi_z = -G[ (\Zimp + \partial_\nu) f]
 \] 
 since $\phi_z$ is the far field of the outgoing solution $G_k(x-z)$ to the Helmholtz equation outside $D$.  Hence $\phi_z \in \mc{R}(G)$.
 
 Conversely, assume that $z \not \in D$ and let us prove by contradiction that $\phi_z \not \in \mc{R}(G)$.  Assume that there exists $f \in \esplambda^*$ such that
 \[
 \phi_z = Gf.
 \]
 Then the solution to \eqref{pbdiff} coincides with the near field associated
 with $\phi_z$, which is $G_k(x-z)$, in the domain $\Om \setminus \{z\}$. This
 contradicts the singular behavior of $G_k$ at point $0$ that prevents this
 function to be locally $H^1$ in the neighborhood of $z$. 
 \end{proof}

Now we will prove that the factorization \eqref{eq:facto} satisfies the
hypothesis of the Factorization Theorem (see \cite[Theorem 2.15]{kirsch2008factorization}). We mainly have to prove that $G$ is injective with dense range, that $-T^*$ satisfied good sign properties and that its imaginary part is strictly positive.
\begin{lemma}
	\label{Hcomp}
  The operator $G:\esplambda^*\mapsto L^2\left(S^{d-1}\right) $ is injective and compact with dense range.
\end{lemma}
\begin{proof}
Thanks to Proposition \ref{extension}, $(\Zimp+n_e)^{-1}$ is an isomorphism from $\esplambda^*$ into $\trace$.
Moreover from \cite[Lemma 1.13]{kirsch2008factorization} we get that $A^\infty:\trace \mapsto L^2(S^{d-1})$ is injective and compact with dense range. 
Then using factorization \eqref{factoG} we obtain that  $G:\esplambda^*\mapsto L^2\left(S^{d-1}\right)$  is also injective compact with dense range. 
\end{proof}

\begin{lemma}
  \label{coerT}
If one of the two assumptions is fulfilled:
\begin{enumerate}
\item the space $\vgamma$ is compactly embedded into $\trace$,
\item the space $\trace$ is compactly embedded into $\vgamma$,
\end{enumerate}
then the real part of $r T$ writes as 
\begin{equation*}
  C_T+K_T,
\end{equation*}
with $r=1$ if assumption (i) is fulfilled and $r=-1$ if assumption (ii) is fulfilled where $C_T:\esplambda \mapsto \esplambda^*$
is a self-adjoint and coercive operator i.e.
 \begin{equation*}
 \exists c_T>0, \ \langle C_T x,x \rangle \geq c_T\| x\|^2_{\esplambda} \ \forall x \in \esplambda.
 \end{equation*}
Moreover $K_T:\esplambda \mapsto \esplambda^*$ and $\Im(T):\esplambda \mapsto \esplambda^*$ are compact operators.
\end{lemma}
\begin{proof}
The operators  $\tracepsi:\traced \mapsto \trace$ and $-\normalepdi:\trace \mapsto \traced$ are self-adjoint and coercive (see proofs of Lemma 1.14 and Theorem 1.26 of  \cite{ kirsch2008factorization}), i.e., there exists two strictly positive constants $c_{\SLI}$ and $c_{\DLI}$ such that:
\begin{equation*}
  \begin{cases}
    \langle \tracepsi x,x \rangle \geq c_{\SLI}\| x\|^2_{\traced} \ \forall x \in \traced,\\
    \langle -\normalepdi x,x \rangle \geq c_{\DLI}\| x\|^2_{\trace} \ \forall x \in \trace.
  \end{cases}
\end{equation*}

Assume that $\vgamma$ is compactly embedded into $\trace$. From the coercivity results on $\tracepsi$ and $\normalepdi$, we get that
 $C_T:=\Zimp \tracepsi \Zimp^*-\normalepdi:\esplambda\mapsto \esplambda^*$ is self-adjoint and coercive.
Now let us prove that:
\begin{equation*}
\mathcal{K}:=  \Zimp( \traceps-	\tracepsi) \Zimp^* +\normalepd+\normalepdi +\Zimp \tracepd +\normaleps \Zimp^*,
\end{equation*}
is compact as an operator from $\esplambda$ onto $\esplambda^*$. Since $(\traceps-\tracepsi) : \traced \mapsto \trace$ is compact we obtain  that $\Zimp (\traceps-\tracepsi)\Zimp^*:\esplambda \mapsto \esplambda^*$ is compact.
The operators
\begin{equation*}
  \begin{cases}
    \normalepd+\normalepdi:\trace \mapsto \traced,\\
    \Zimp \tracepd:\trace \mapsto \esplambda^*,\\
    \normaleps \Zimp^*:\esplambda \mapsto \traced,
  \end{cases}
\end{equation*}
are bounded operators and using that $\esplambda$ is compactly embedded into $\trace$ (Lemma \ref{Lambdabien}) we can conclude that $\mathcal{K} : \esplambda\mapsto \esplambda^*$ is compact and so are its imaginary and real parts.

Now assume that $\trace$ is compactly embedded into $\vgamma$. Thanks to Remark \ref{rem:inclusion} we have that $\esplambda=\trace$ and then $\Zimp,\Zimp^*:\trace \mapsto \traced$ are compact operator.
Define $C_T=-\normalepdi$, we recall that this operator is self adjoints and coercive and that $\tracepd:\trace \mapsto \trace$ and $\normaleps: \traced \mapsto \traced$ are bounded operator.
Furthermore $\normalepd-\normalepdi:\trace \mapsto \traced$ is compact 
 and then the operator
\begin{equation*}
  \mathcal{K}:=\Zimp \traceps \Zimp^* +\Zimp \tracepd + \normaleps \Zimp^*+\normalepd-\normalepdi,
 \end{equation*}
 is compact as an operator from $\trace$ onto $\traced$ and thus the same goes for its real and imaginary parts.
 \end{proof}
We still need to prove that the imaginary part of $-T^*$ is strictly positive on  $\overline{\mathcal{R}(G^*)}$. To do so, we have to avoid  some special values of $k$ for which this may not be true.
\begin{definition}
 We say that  $k^2$ is an eigenvalue of $- \Delta$ associated with the impedance operator 
$\Zimp$ if there is a nonzero  $ u \in \{ v \in H^1(D) ,\ \gamma_0 v \in \espace \}$ solution to
\begin{equation}
\label{eq:eigenX}
\begin{cases}
 \Delta u + k^2 u=0 \text{ in } D, \\
\pn u+\Zimp u=0 \text{ on } \Gamma.
\end{cases}
\end{equation}
\end{definition}
\begin{remark}
We remark  that when $\Im(\Zimp) \ : \  \espace \mapsto \espacedual $ is positive definite then there is no real eigenvalue $k^2$ associated with the impedance operator $\Zimp$. Actually, any solution $u \in  \{ v \in H^1(D) ,\ \gamma_0 v \in \espace \}$ to \eqref{eq:eigenX} for $k^2 \in \bR$ satisfies
\[
\langle \Im (\Zimp) u,u \rangle =0
\]
which implies $u=0$ when $\Im(\Zimp)$ is positive definite.
\end{remark}

\begin{lemma}
\label{posT}
If $k^2$ is not an eigenvalue of $- \Delta$ associated with the impedance operator 
$\Zimp^*$ then
$-\Im(T^{*})$ is strictly positive on $\overline{\mathcal{R}(G^*)}$.
\end{lemma}
\begin{proof}
First of all, we can prove similarly to \cite[Theorem 2.5]{kirsch2008factorization} that
\begin{equation}
\label{eq:posF}
\Im(F)=k |\gamma(d)|^2  F^{*}F +R_{\Zimp},
\end{equation}
where $\gamma(2) = \frac{e^{i\pi/4}}{\sqrt{8\pi k}} $ and $\gamma(3)=\frac{1}{4\pi}$, and $R_{\Zimp} : L^2(S^{d-1}) \mapsto L^2(S^{d-1})$ is such that for all $(g,h) \in (L^2(S^{d-1}))^2$  
\[
(R_{\Zimp}g,h)_{L^2(S^{d-1})} = \Im\langle \Zimp V_g,V_h \rangle.
\]  
The functions $V_g:=v^s_g+v_g$ and $V_h:=v^s_h+v_h$ are the total fields associated with the scattering of  the Herglotz incident waves $u_i=v_g$ and $u_i=v_h$ respectively, namely $v^s_g$ (resp. $v^s_h$) satisfies \eqref{pbdiffdeb}-\eqref{eq:radiation} with $u_i = v_g$ (resp. $v_h$). For all $x \in \mathcal{R}(G^*)$ with $x=G^*y$ we have
\begin{align*}
 \Im\langle -T^* x, x\rangle=\Im\langle -T^* G^*y, G^*y\rangle=\Im \langle Fy,y\rangle
\end{align*}
and by using \eqref{eq:posF} together with the fact that $R_{\Zimp}$ is non negative we get:  
\begin{equation}
  \label{eq:posT}
  \Im\langle -T^* x, x\rangle \geq k|\gamma(d)|^2\| F y\|^2
  \geq k|\gamma(d)|^2\| G T^* G^* y\|^2
  \geq  k|\gamma(d)|^2\| G T^* x\|^2.
\end{equation}
Since $G$ and $T^*$ are two bounded operators inequality \eqref{eq:posT} holds for all $x \in \overline{\mathcal{R}(G^*)}$. Hence we simply  have to prove that $GT^{*}$ is injective to obtain the result of the Lemma. 

The operators $A^\infty:\trace \mapsto L^2(S^{d-1})$ and
$(\Zimp+n_e)^{-1}:\esplambda^* \mapsto \trace$ are injective, therefore,
\eqref{factoG}  implies that $G$ is also injective. Hence if $T^{*}$ is injective, $GT^{*}$ will also be injective. Let us prove that in fact $T$ and then $T^{*}$  are isomorphisms. From Lemma \ref{coerT}, the operator $T$ is of  Fredholm type and index zero, hence we simply have to prove that it is injective to deduce that it is an isomorphism. By \eqref{decomposition_hstart}, the injectivity of $T$ is equivalent to the injectivity of $(\Zimp \traceH+\partialH)^*$, take $q  \in \esplambda$ such that $(\Zimp \traceH+\partialH)^* q =0$, expression \eqref{eq:faradj} leads to
\begin{equation*}
\displaystyle \int_{\Gamma} e^{-ik \hat{x}y}\Zimp^*q(y)ds(y)+\int_{\Gamma} \frac{\partial e^{-ik \hat{x}y}}{\partial \nu(y)}q(y)ds(y)=0.
\end{equation*}
The left-hand side of this expression is the far field pattern associated with
\begin{equation*}
  v_+:=\SL(\Zimp^* q) + \DL(q)=0 \quad\text{ in } \Om.
\end{equation*}
Let us also define	
\begin{equation*}
  v_-:=\SL(\Zimp^* q) + \DL(q) \quad\text{ in } D.
\end{equation*}
Rellich's Lemma  implies that $v_+=0$ outside $D$ and  thanks to the jump conditions for the single and double layer potentials (see \cite{McL00} for example) we have on $\Gamma$:
\begin{equation*}
\begin{cases}
 v_-=v_--v_+=-q, \\
 \pn v_-=\pn v_-- \pn v_+=\Zimp^* q.
\end{cases}
\end{equation*}
% Hypothesis (H2)  on the operator $\Zimp$ leads to
% \begin{equation*}
% \begin{cases}
%  v_-=-q, \\
%  \pn v_-=\overline{\Zimp \overline{q}}
% \end{cases}
% \end{equation*}
% which implies by conjugation:
% \begin{equation*}
%  \pn \overline{v_-}+\Zimp \overline{v_-}=0.
% \end{equation*}
% Finally, $\overline{v_-}$ satisfies the Helmholtz' equation inside $D$ with
% homogeneous generalized impedance boundary condition thus, 
The  assumption on $k^2$ implies that $v_-=0$ inside $D$ and therefore
$q=0$. We finally get that $T^{*}$ is an isomorphism, which concludes the proof.
\end{proof}
We are now in position to  state the main Theorem of this paper.
\begin{theorem}
\label{caracterisation}
If one of the two assumptions is fulfilled:
\begin{enumerate}
\item the space $\vgamma$ is compactly embedded into $\trace$,
\item the space $\trace$ is compactly embedded into $\vgamma$,
\end{enumerate}
then provided that $k^2$ is not an eigenvalue of $- \Delta$ associated with the impedance operator 
$\Zimp^*$ we have:
\begin{equation}
\label{rangegal}
 \mathcal{R}(G)=\mathcal{R}(F_{\#}^{\frac{1}{2}})
\end{equation}
 and 
\begin{equation}
  \label{eq:caractcont}
  z \in D \Longleftrightarrow \phi_z \in \mathcal{R}(F_{\#}^{\frac{1}{2}}).
\end{equation}
\end{theorem}
\begin{proof}
Under these assumptions the results of Lemmas \ref{Hcomp}, \ref{coerT}  and
\ref{posT}  hold which means that the factorization of the far field operator satisfies
all the requirements of \cite[Theorem 2.15]{kirsch2008factorization}. Thus $\mathcal{R}(G)=\mc{R}(F_{\#}^{\frac{1}{2}})$ and we can conclude using Lemma \ref{LeImG}. 
\end{proof}
\begin{remark}
\label{re:dipole}
  Identity \eqref{rangegal} and the decomposition $G=A^\infty(Z_k+n_e)^{-1}$ implies that we can replace the family $(\phi_z)_{z \in \mathbb{R}^d}$ by any family of functions $(\psi_z)_{z \in \mathbb{R}^d}$ satisfying:
  \begin{equation}
    \label{eq:remark}
    z \in D \Longleftrightarrow \psi_z \in \mathcal{R}(A^\infty).
  \end{equation}
An example of these functions is given  by  the far field pattern of a dipole located at
point $z$  with polarization $p$ defined by $ \psi_{z,p}(\hat x) = p \cdot
\hat{x} \phi_z(\hat x) $. These are the far fields associated with $x \mapsto p\cdot\nabla_x G_k(x-z)$. 
\end{remark}
\begin{remark}
\label{re:fail}
 In the intermediate case, when none of the compact embeddings ($\vgamma \subset \trace$ or $\trace \subset \vgamma$) hold, the principal part of operator $T$ fails to be positive. As a matter of fact we are not able to  link the range of $G$ with the range of $F_\#^{\frac{1}{2}}$.
\end{remark}

To complement this work we should verify that most of the time $k^{2}$ is not an eigenvalue of $- \Delta$ associated with the impedance operator 
$\Zimp^*$. 
\begin{theorem}
\label{discretvp}
Assume that $\Zimp^*$ depends analytically on $k\in \bC$. Under one of the assumptions of Theorem \ref{caracterisation}, the set of eigenvalues of $- \Delta$ associated with the impedance operator $\Zimp^*$ is discrete with no finite accumulation point. 
\end{theorem}
\begin{proof}
In order to prove that, we use the Analytic Fredholm Theorem (\cite[Theorem
8.19]{colton1998inverse}). If $-k^2$ is not a Dirichlet eigenvalue  in $D$ we
can define the interior Dirichlet-to-Neumann map $n_i(k): \trace \mapsto
\traced$ for $g \in \trace$ by $n_i(k)g=\pn u$ where $u$ is the unique solution
to the Helmholtz' equation in $D$ with $u=g$ on $\Gamma$. If $\Zimp^*+n_i(k)$
is injective, then $-k^2$ is not an eigenvalue associated with $\Zimp^*$.  In
the remaining of the proof, we show that $\Zimp^*  +n_i(k)$ fails to be injective only for $k$ in a discrete set of $\bC$. 
 
Let $k_0$ be a complex number   such that:
$ k^2_0=i$,
and assume that the embedding of $\vgamma$ into $\trace$ is compact. Then,
\begin{equation}
  \label{eq:decomposition:a}
  \Zimp^*+n_i(k)=T_k+n_i(k)-n_i(k_0),
\end{equation}
where $T_k:\espace \mapsto \espacedual$ is defined by
\begin{equation*}
      T_k :=\Zimp^* + n_i(k_0).
\end{equation*}
Let us prove that $T_k$ is an isomorphism. Following the lines of the proof of
points 2 and 3 of 
Corollary \ref{existence} one easily observe that $T_k$ is Fredholm with index 0.  Let us prove that $T_k:\espace \mapsto \espacedual$ is injective. Take $g \in \espace$ such that $T_k g=0$ and $u_g \in H^1(D)$ such that
\begin{equation*}
  \begin{cases}
  \Delta u_g+ k_0^2 u_g=0 \text{ in } D,\\
  u_g=g \text{ on } \Gamma. 
  \end{cases}
\end{equation*}
Then $\partial_\nu u_g = n_i(k_0)g = - \Zimp^* u_g$ and Green's formula yields:
\begin{equation*}
 \int_D |\nabla u_g|^2 dx- k_0^2 \int_D |u_g|^2 dx + \langle\Zimp^* u_g,u_g\rangle=0. 
\end{equation*}
By taking the imaginary part of this last equation we have:
\begin{equation*}
 \Im(k^2_0) \int_D |u_g|^2 dx - \Im\langle\Zimp^* u_g,u_g\rangle=0.
\end{equation*}
Since we supposed that  $\Im(\Zimp) \geq 0$ and $k_0^2=i$, we get
\begin{equation*}
\int_D |u_g|^2 dx=0, 
\end{equation*}
which gives $g=0$ and $T_k$ is an isomorphism. From \eqref{eq:decomposition:a} we have:
\begin{equation*}
  \Zimp+n_i(k)=T_k(I+T_k^{-1}(n_i(k)-n_i(k_0)))
\end{equation*}
and then $\Zimp+n_i(k):\espace \mapsto \espacedual$ is injective if and only if the operator
\begin{equation*}
  I+T_k^{-1}(n_i(k)-n_i(k_0)):\espace \mapsto \espace,
\end{equation*}
is injective. The parameter dependent family of operators $n_i(k)$ depends analytically on $k$  over the complement of the Dirichlet eigenvalue  in $D$ and we supposed that
$\Zimp^*$ depends analytically on $k$ then $T_k^{-1}(n_i(k)-n_i(k_0))$ depends analytically on $k$. Since this operator is bounded from $\trace \mapsto \espace$ this operator is a compact operator from
$\espace \mapsto \espace$. Thus we can apply the the analytic Fredholm Theorem to
\begin{equation*}
  I+T_k^{-1}(n_i(k)-n_i(k_0)).
\end{equation*}
For $k=k_0$ this operator is  injective and then one can conclude that $  I+T_k^{-1}(n_i(k)-n_i(k_0))$ is injective except for a discrete a discrete set of $k$. Thus the set of the eigenvalues of $- \Delta$ associated with the impedance condition  $\Zimp^*$ is discrete with no finite accumulation point  

\end{proof}
\section{Numerical tests for second order surface operators}
Motivated by models for thin coatings we consider in this section impedance
operators of the form
\begin{equation}
\label{eq:impoperator}
  \Zimp=\divt \mu \nabla_\Gamma - \lambda
\end{equation}
where $(\mu,\lambda) \in (L^\infty(\Gamma))^2$ depend analytically on $k$ with
non positive imaginary parts and where the real part of  $\mu$ is positive
definite or negative definite. Then, Theorems \ref{caracterisation} and \ref{discretvp} apply and except for a discrete set of wave numbers $k$, the following equation
\begin{equation}
\label{caracterisationequ}
 F_\#^{\frac{1}{2}}g_z=\phi_z,
\end{equation}
has a solution if and only if $z \in D$ whenever $\phi_z$ is a valid right-hand
side to apply Theorem \ref{caracterisation} (see Remark \ref{re:dipole}). We
shall 
give in the following some two dimensional numerical experiments that show how this test behaves
when the far field data are given by an obstacle characterized by a generalized impedance boundary condition of the form  \eqref{eq:impoperator}.

\subsection{Numerical setup and regularization procedure}
Here we identify the unit sphere $S^1$ of $\mathbb{R}^2$ with the interval $] 0,2 \pi[$. The data set we consider is not the operator $F$ itself but a discrete version that we represent by the matrix $\mathbb{U}_n:=(u^{n}_{ij})_{1 \leq i \leq n,1 \leq j \leq n}$ where $u^{n}_{ij}$ is an  approximation of $u^\infty \left(\frac{2\pi (i-1)}{n},\frac{2\pi (j-1)}{n}\right)$. For $\hat \theta \in  S^1$, $u^\infty(.,\hat{\theta})$ is the far field associated with $u_s$ which is the unique solution of \eqref{pbdiffdeb}-\eqref{eq:radiation} with $u_i$ defined for $x \in \mathbb{R}^2$ by $u_i(x):=e^{ik \hat{\theta} \cdot x}$.

In practice, the matrix  $\mathbb{U}_n$ contains the values of a synthetic far field computed with a finite element method. We also contaminate these data by adding some random noise and build a noisy far field matrix $\mathbb{U}^\delta_n:=(u^{n,\delta}_{ij})_{1 \leq i \leq n,1 \leq j \leq n}$ where
\begin{equation*}
  u^{n,\delta}_{ij}=u^{n}_{ij}(1+\eta (X^{ij}_1+iX^{ij}_2))
\end{equation*}
for $\eta>0$ and $(X^{ij}_k)_{1 \leq i \leq n, 1 \leq j \leq n,k \in \{ 1,2\}}$ are uniform random variables on $[-1,1]$.
We denote by $\Phi_z^n \in \mathbb{C}^n$ the vector defined $1 \leq i \leq n$, $\Phi^n_z(i)=\phi_z(\frac{2 \pi i}{n})$ where $\phi_z$ is a valid right-hand side to apply Theorem \ref{caracterisation} (see Remark \ref{re:dipole}). We define the  real symmetric matrix $(\mathbb{U}^\delta_n)_\#:=|\Re(\mathbb{U}^\delta_n)|+ |\Im(\mathbb{U}^\delta_n)|$ and denote by $(e^n_i,\lambda^n_i)_{1\leq i \leq n}$ its eigenvectors and eigenvalues. Since this matrix is positive, its eigenvalues are positive and we  assume in addition that they are non zero. Then, Theorem \ref{caracterisation} and Picard's criterion suggest to use the function
\begin{equation}
\label{eq:Picard}
  z \mapsto \left(\sum\limits_{j=1}^n \frac{1}{\lambda_j^n}|(e_j^n,\phi_z^n)|^2\right)^{-1}
\end{equation}
as an indicator function for the obstacle $D$ since it should have small values for any $z$ outside the obstacle $D$ and greater values for $z$ inside $D$.
Since the far field operator $F$ is compact, the values $\lambda_j^n$ are very small for large $j$ in $\mathbb{N}$. When considering noisy data, it occurs that  $|(e_j^n,\phi_z^n)|$ is not small for large $j$ and as a consequence criterion \eqref{eq:Picard} could be small even if $z$ is in $D$. Therefore, we prefer to use the regularized function:
\begin{equation}
 \label{def:g_z}
  w_n(z):=\left(\sum\limits_{j=1}^n \frac{\lambda_j^n}{(\alpha_n+\lambda_j^n)^2}|(e_j^n,\phi_z^n)|^2\right)^{-1},
\end{equation}
where $\alpha_n$ is the regularization coefficient.
We remark that $w_n(z)$ is nothing but the inverse of the squared norm of the  unique solution to the discrete problem with Tikhonov regularization
\[
 (\alpha_n I +(\mathbb{U}^\delta_n)_\#)g_n(z) = (\mathbb{U}^\delta)_\#^{\frac{1}{2}} \phi_z^n.
\]
The regularization coefficient  $\alpha_n$ is chosen by using the Morozov discrepancy principle which is: choose  $\alpha_n$ as the unique solution of find $\alpha \in (0,\delta_n\lambda^n_1 ]$ such that
\[
\|(\mathbb{U}^\delta_n)_\#^{\frac{1}{2}} g_n(z) - \phi^n_z \|_{L^2(S^1)} = \delta_n \|g_n(z)\|_{L^2(S^1)}
\]
or equivalently
\begin{equation*}
 % \label{morozov}
  \sum\limits_{j=1}^n \frac{\alpha^2-\delta_n^2 |\lambda_j^n|}{|\alpha+\lambda_j^n|^2}|(e_j^n,\phi_z^n)|^2=0,
\end{equation*}
where $\delta_n>0$ is a bound on the noise on the operator $F_\#^{1/2}$ and the right hand side $\phi_z$.
It has been proven in \cite{Kir98} (see also \cite{TiGoStYa95}) that the function $w_n(z)$ is an indicator function of $D$ when $n$ tends to infinity and the noise $\delta_n$ to $0$.
\subsection{Numerical experiments}
Let $w_n(z)$ be the function given by \eqref{def:g_z} whith the  test function
$\phi_z = \psi_z / \|\psi_z\|_{L^2}$ with
$\phi_z(\hat{x}) =  e^{-ik\hat{x}\cdot z}$   and
$w_{n,\text{dipole}}^\theta(z)$ the function given by \eqref{def:g_z} when
$\phi_z = \psi_z / \|\psi_z\|_{L^2}$ with $\psi_{z}(\hat{x}) = p(\theta) \cdot \hat{x}  e^{-ik\hat{x}\cdot z}$, with
$p(\theta) = (\cos(\theta), \sin(\theta))^t$. We choose to take $4$ directions
of the polarization $p$ and we define
\[
w_{n,\text{dipole}}(z):=\min\limits_{ \theta \in \{0,\pi/4,3\pi/4,\pi\}} w_{n,\text{dipole}}^\theta(z).
\]
Remark that thanks to the linearity of the far field equation, taking $\theta$
in the interval $[0,\pi]$ is equivalent to take it in $[0,2\pi]$.  We finally plot the values of
\begin{equation*}
W_n(z) := w_n(z)+w_{n,\text{dipole}}(z)
\end{equation*}
in order to determine the indicator function of the scatterer. We refer to
\cite{boukari:hal-00768729,benhassen:hal-00743816} for a discussion on this choice of indicator function based on the
analysis of impedance boundary conditions for cracks.  Numerical illustration
of the importance of this choice is given by Figure \ref{fig:Comp} where we
compare the plots of  $w_n$, $w_{n,\text{dipole}}$ and $W_n$ for $\mu=0.1$. The
reconstruction provided by $W_n$ is much better than the two others.
  \begin{figure}
  \centering
  \begin{tabular}{ccc}
\includegraphics[width=.3\textwidth]{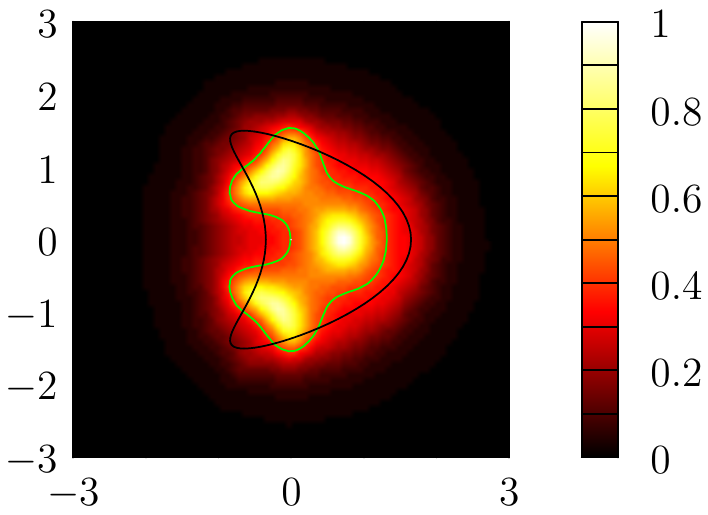}&
\includegraphics[width=.3\textwidth]{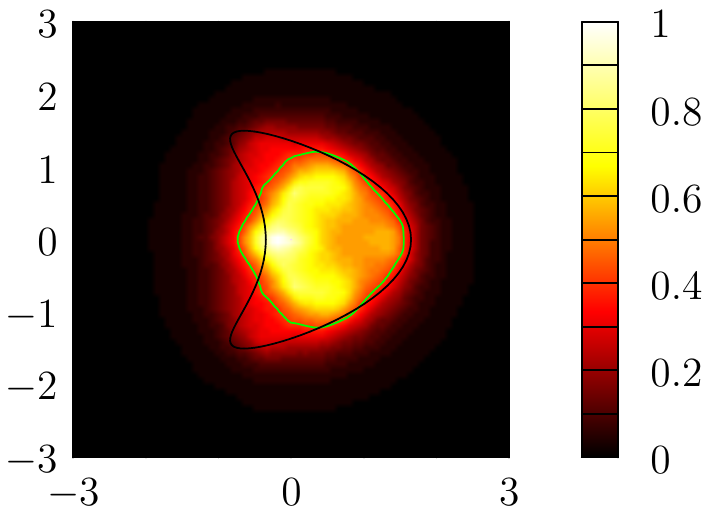} \hfill &
  \includegraphics[width=.3\textwidth]{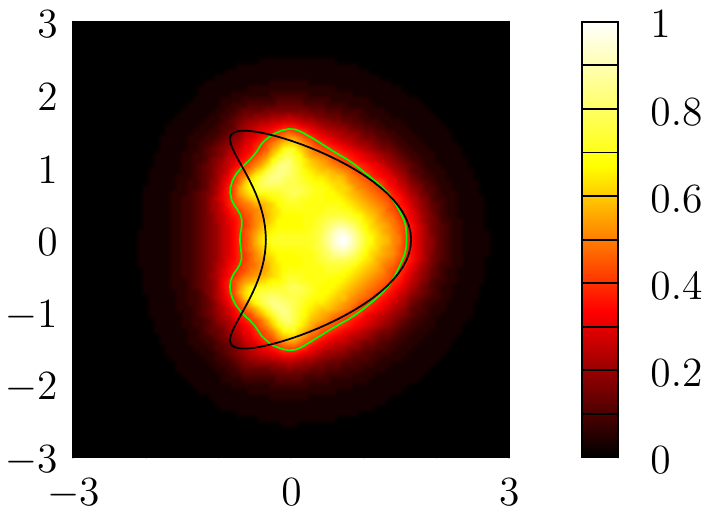} \\
 (A) \small{$w_n$} &\small{(B) $w_{n,\text{dipole}}$}&\small{(C) $W_n$}
  \end{tabular}
\caption{Comparison of different right-hand sides, with $\mu=0.1$, $\lambda=0$, $k=2$ and no noise.}
\label{fig:Comp}
\end{figure}
The searched domain (the black line on the plot) is given by
\begin{equation}
\label{eq:gee}
 \partial D=\{ (\cos(t)+0.65\cos(2t),1.5\sin(t)),\ 0\leq t \leq 2 \pi \}.
\end{equation}
The green line represents the isoline of the function $W_n(z)$ that visually
fits well the unknown shape $\partial D$. The size of $D$ is of the order of
one wavelength  when the wavenumber $k$ is equal to $2$.  The sampling domain
is the square   $[-3,3]^2$ which is discretized with  $80\times80$ points
$z$. Finally, we take $n=50$, which means that we send $50$ incident waves
uniformly distributed on the unit circle and we observe the far field at all
these directions.  For noisy data, we take $\eta=1\%$. We restrict ourselves to
the cases $\lambda = 0$ but the results for $\lambda \neq 0$ are similar.   

Figures  \ref{fig:noise} and \ref{fig:kitenoise} show the influence of the
generalized impedance coefficient $\mu$ on two different geometries: the first
one is convex and is an ellipse of semi-minor axis 1 and semi-major axis 2, the
second is the non-convex obstacle (Kite shape) given by \eqref{eq:gee}. For
large values and small values  of $\mu$ ($\mu=10$ and $\mu=0.01$) the
reconstruction is quite satisfactory (see Figures
\ref{fig:noise}(C), \ref{fig:kitenoise}(C),
\ref{fig:noise}(A) and \ref{fig:kitenoise}(A)) while for an
intermediate value ($\mu=0.1$)  the reconstruction is poorer (see Figures
\ref{fig:noise}(B) and \ref{fig:kitenoise}(B)). Let us also
mention that the noiseless reconstruction (Figure \ref{fig:Comp}(C)) for $\mu=0.1$
is quite accurate which means that in this case the reconstruction is sensitive to the noise. A possible explanation is that the test functions $\phi_z$ we use are not well-suited to this case. Figure \ref{fig:k5} shows the influence of the frequency, these figures should be compared to the Figure \ref{fig:kitenoise}(B) and, as we would expect, we increase the precision of the reconstruction by increasing the frequency.

\begin{figure}
\begin{tabular}{ccc}
\includegraphics[width=.3\textwidth]{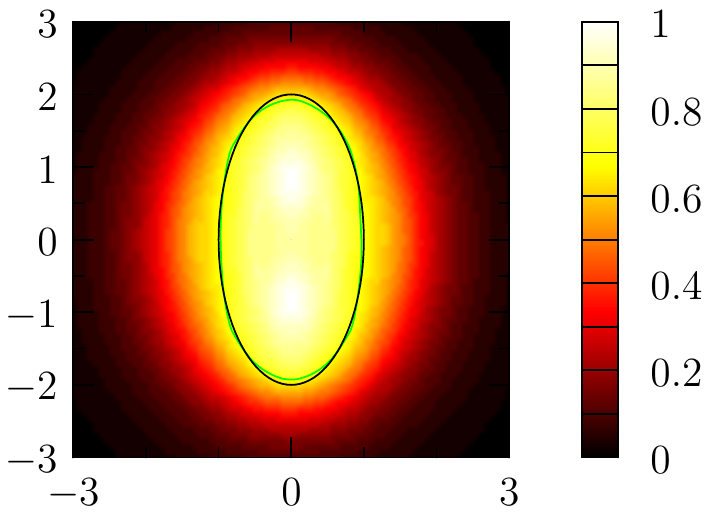}&
\includegraphics[width=.3\textwidth]{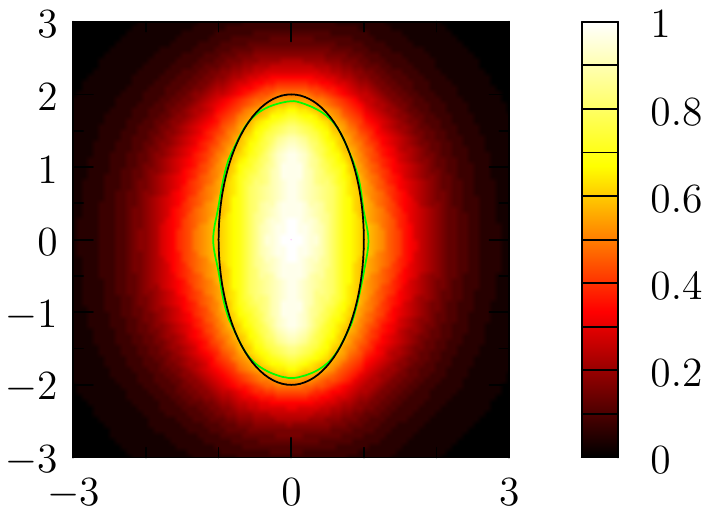}\label{fig:mumoyenellipse2noise}&
\includegraphics[width=.3\textwidth]{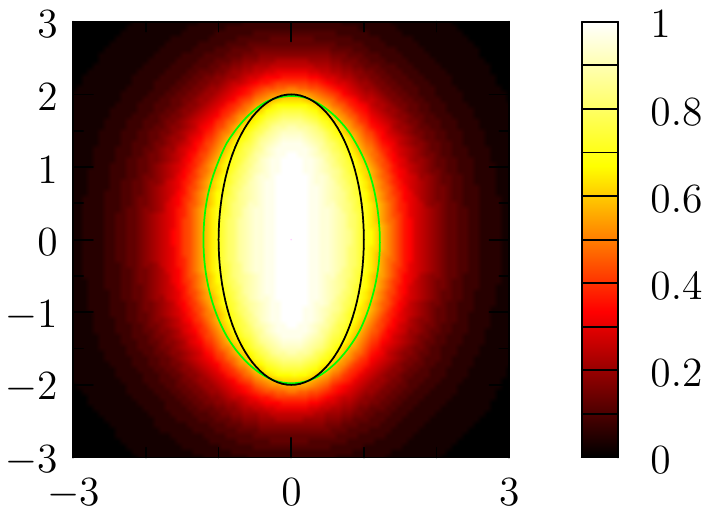}\label{fig:mugrandellipse2noise} \\
\small{(A) ellipse, $\mu=0.01$ }&\small{(B) ellipse, $\mu=0.1$}&\small{(C) ellipse, $\mu=10$}
\end{tabular}
\caption{Reconstruction of a convex geometry for several $\mu$, $k=2$ with 1\% of noise.}
\label{fig:noise}
\end{figure}

\begin{figure}
\begin{tabular}{ccc}
\includegraphics[width=.3\textwidth]{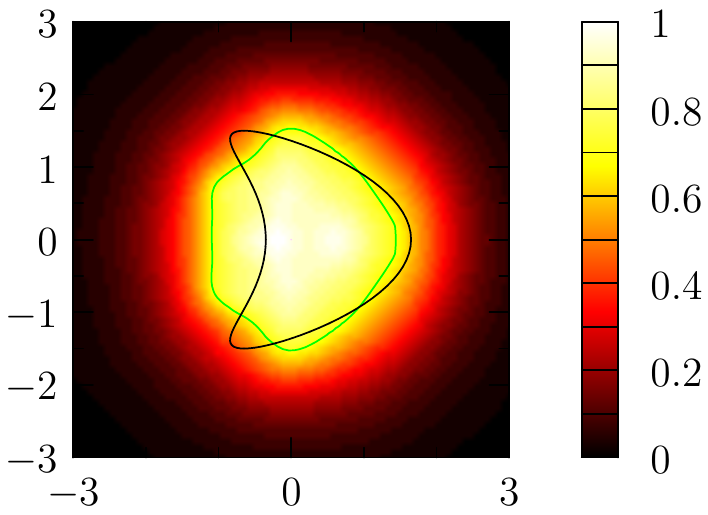}\hfill &
\includegraphics[width=.3\textwidth]{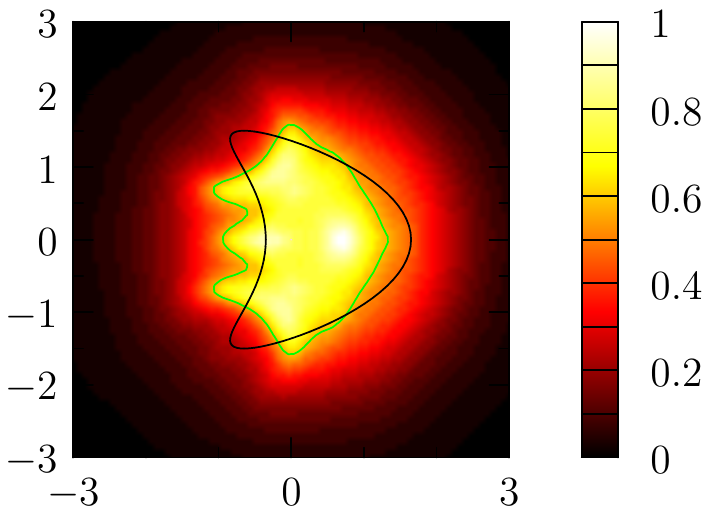}&
\includegraphics[width=.3\textwidth]{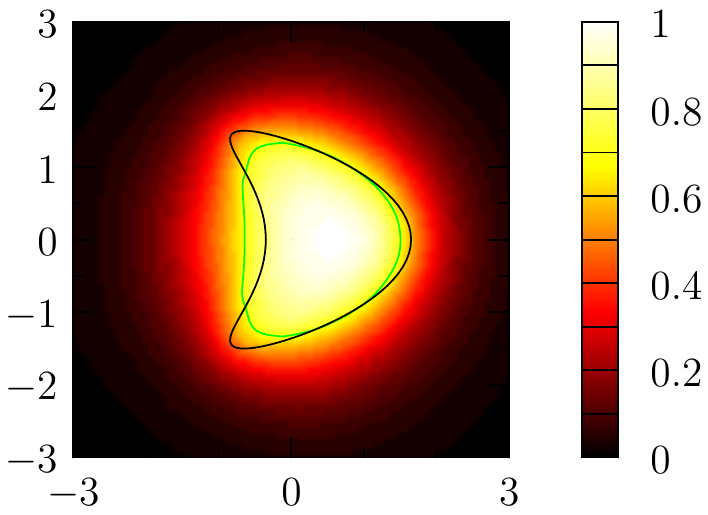}\\
\small{(A) $\mu=0.01$}&\small{(B) $\mu=0.1$}&\small{(C) $\mu=10$}
\end{tabular}
\caption{Reconstruction of a non-convex geometry for several $\mu$, $k=2$ with 1\% of noise.}
\label{fig:kitenoise}
\end{figure}

\begin{figure}
\begin{tabular}{ccc}
\includegraphics[width=.3\textwidth]{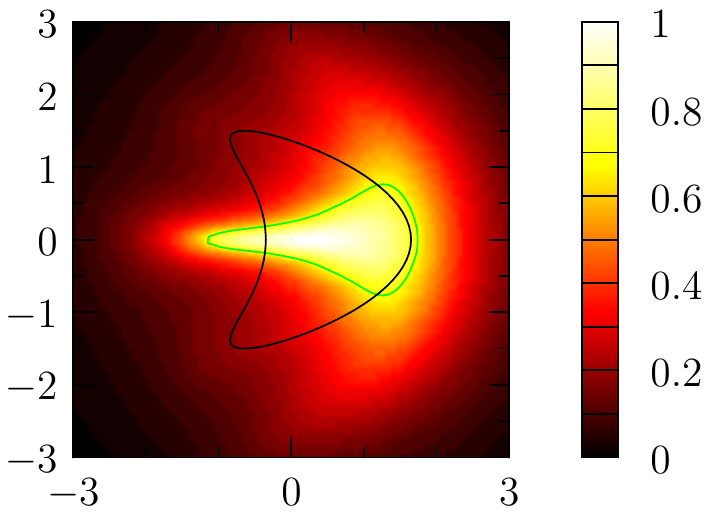}&
\includegraphics[width=.3\textwidth]{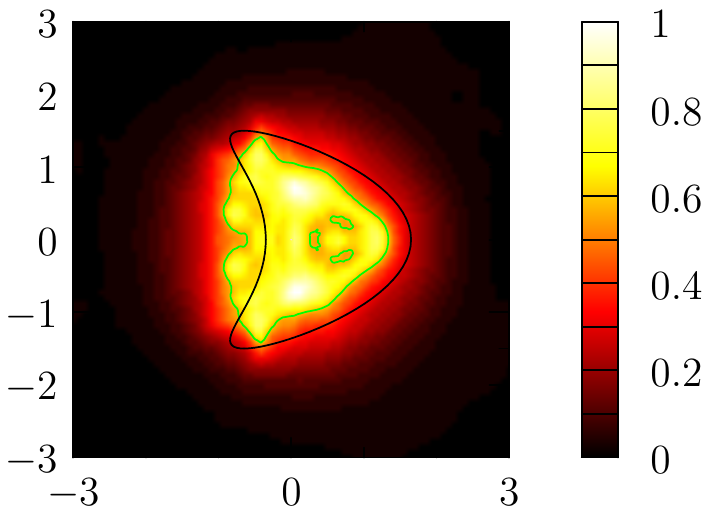}&
\includegraphics[width=.3\textwidth]{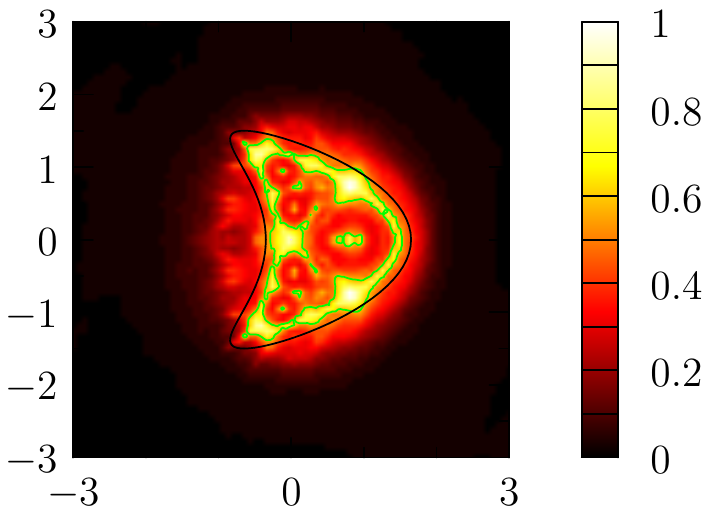}\\
\small{(A) $k=1$}&\small{(B) $k=4$}&\small{(C) $k=8$}
\end{tabular}
\caption{Reconstruction of a non-convex geometry for several $k$, $\mu=0.1$ with 1\% of noise.}
\label{fig:k5}
\end{figure}
\clearpage

\section*{Appendix}
\begin{proof}[Proof of Proposition \ref{fredne}]
Let us denote by $B_r$ an open ball of radius $r$ which contains $\overline{D}$.  The exterior Dirichlet-to-Neumann map on $\partial B_{r}$, $S_r:\traces \mapsto \tracesd$ is defined for $g \in \traces$ by $S_rg=\partial_r u^e|_{\partial B_r}$ where $u^e$ is the outgoing solution of the Helmholtz' equation outside $B_r$ and $u^e=g$ on $\partial B_r$. Thanks to the appendix of \cite{Koy03} for $d=2$ and \cite[Theorem 2.6.4]{Ned01} for $d=3$ we recall that the operator $S_r$ satisfies
\begin{align}
  \label{eq:s_rre}
  \Re \langle S_r u,u \rangle &\leq 0,\ \forall u \in \traces,\\
  \label{eq:s_rim}
  \Im \langle S_r u,u \rangle&\geq0,\ \forall u \in \traces .
\end{align}
  Let us denote $\Omega_r:=(\mathbb{R}^d \setminus \overline{D})\cap B_r$ and  define the bounded operator $A_r:\trace \mapsto H^1(\Omega_r)$
for $g \in \trace$ by $A_r g:=u_r$ where $u_r$ is the unique solution of:
\begin{equation}	
\label{pbreduit}
 	\begin{cases}\Delta u_r+k^2 u_r=0 \text{ in } \Omega_r,\\
	u_r=g \sg, \\
	\partial_r u_r =S_r u_r \text{ on } \partial B_r.
	\end{cases}
\end{equation}
Let us introduce $C_e: \trace \mapsto \traced$ and $K_e:\trace \mapsto \traced$ defined by
\begin{equation*}
\langle C_e f,g \rangle= (A_rf,A_rg)_{H^1(\Omega_r)}-\langle S_r f,g\rangle,
\end{equation*}
and
\begin{equation*}
\langle K_e f,g \rangle = (k^2+1)\int_{\Omega_r} A_rf \overline{A_rg} d x,
\end{equation*}
for all $(f,g) \in \trace\times \trace$. Using Green's formula and the Helmholtz' equation satisfied by $A_{r}f$ and $A_{r}g$ in $\Omega_r$, we prove that
\begin{equation*}
 n_e=-C_e+K_e.
\end{equation*}
Let us prove that $C_e$ has a coercive real part and that $K_{e}$ is compact. Remark that 
\begin{equation}
\label{eq:Ce}
 \Re\langle C_e f,f\rangle=\| A_rf \|^2_{H^1(\Omega_r)}-\Re\langle S_r f,f \rangle,
\end{equation}
but $ \Re \langle S_r f,f \rangle \leq 0$ by \eqref{eq:s_rre}. The range of $A_r$ is given by:
\begin{equation*}
 \mathcal{R}(A_r)=\{u \in H^1(\Omega_r) ,\ \Delta u +k^2 u=0 \text{ and } S_r u=\pn u \text{ on } \partial B_r \},
\end{equation*}
and this  space is a closed subspace of $H^1(\Omega_r)$.  Hence, since $A_r$ is injective  there exists $C>0$ such that
\begin{equation*}
 	\|A_rf\|_{H^1(\Omega_r)} \geq C \| f\|_{\trace} ,\ \forall f \in \trace.
\end{equation*}
In regard of \eqref{eq:Ce}, we deduce that $C_{e}$ has a coercive real part. We also have 
\[
\Im \langle C_e f,g\rangle = -\Im\langle S_r  f,g\rangle \leq 0
\]
by \eqref{eq:s_rim}.

To conclude remark that $A_r:\trace \mapsto H^1(B_r)$ is bounded, and since the embedding of $H^{1}(\Omega_{r})$ into $L^{2}(\Omega_{r})$ is compact, $K_e: \trace \mapsto \traced$ is compact. 
\end{proof}

\begin{proof}[Proof of lemma \ref{Lambdabien}]
Let us prove that the injection of $\Lambda(\Gamma)$ into $\espace$ is dense.  Take $x \in \espace$ and $\epsilon>0$. The space $L^2(\Gamma)$ is densely embedded in $\espacedual$ and then there exists $y_\epsilon \in L^2(\Gamma)$ such that:
\begin{equation*}
   \| (\Zimp+n_e)^*x -y_\epsilon\|_{\espacedual} \leq \frac{\epsilon}{\| \left(\Zimp+n_e)^{-1}\right)^{*} \|_{\mathcal{L}\left(\espacedual,\espace\right)}}
\end{equation*}
where for any Hilbert space $E$, $\mc{L}(E,E^*)$ is the space of linear and bounded applications from $E$ to $E^*$.
Thus we have for $z_\epsilon := \left((\Zimp+n_e)^{-1} \right)^*y_\epsilon$
\begin{equation*}
  \|x-z_\epsilon\|_{\espace} \leq \| \left((\Zimp+n_e)^{-1}\right)^* \|_{\mathcal{L}\left(\espacedual,\espace\right)} \| (\Zimp+n_e)^*x -y_\epsilon\|_{\espacedual},
\end{equation*}
and we get
\begin{equation*}
  \|x- z_\epsilon\|_{\espace} \leq \epsilon.
\end{equation*}
But $z_\epsilon \in \espace$ since $y_\epsilon \in \ldeux$ and we have
\[
\Zimp^* z_\epsilon = (\Zimp +n_e)^*z_\epsilon -n_e^* z_\epsilon = y_\ve - n_e^*z_\epsilon \in \trace.
\]
Hence we obtain the density result.

Assume that $\vgamma$ is compactly embedded into $\trace$ and let us prove that $\esplambda$ is also compactly embedded into $\espace$. Let $(x_n)_{n \in \mathbb{N}}$ be a bounded sequence in $\esplambda$.
We have
\begin{align*}
&  \|(\Zimp +n_e)^* x_n\|^2_{\traced} \leq \left(\|\Zimp^* x_n\|_{\traced}+\|n_e^* x_n\|_{\traced} \right)^2,\\
  &\leq 2\max\left(1,\|n_e\|_{\mathcal{L}\left(\trace,\traced\right)} \right) \left(\|\Zimp^* x_n\|^2_{\traced}+\| x_n\|^2_{\trace} \right),\\
\end{align*}
and then $(\Zimp +n_e)^* x_n$ is bounded in $\traced$.
Using that $\traced$ is compactly embedded into $\espacedual$ we get that (up to a subsequence)  $(\Zimp +n_e)^* x_n$ converges into $\espacedual$.
Thus using that $\left((\Zimp +n_e)^*\right)^{-1}:\espacedual \mapsto \espace$ is bounded we get that $x_n$ (up to a subsequence) converges into $\espace$.

 \end{proof}

\bibliographystyle{plain}

\bibliography{biblio}

\end{document}